\documentclass{article}

\usepackage{geometry}
\usepackage{amsthm}
\usepackage{amsmath}
\usepackage{amssymb}
\usepackage{mathtools}
\usepackage{subfigure}

\newcommand{\bb}{\mathbb}
\newcommand{\mc}{\mathcal}
\newcommand{\where}{\ |\ }
\newcommand{\abs}[1]{\left\lvert #1 \right\rvert}
\newcommand{\minus}{\,\backslash\,}

\DeclareMathOperator{\Vol}{Vol}

\newtheorem{thm}{Theorem}[section]
\newtheorem{prop}[thm]{Proposition}

\newtheorem{cor}[thm]{Corollary}

\theoremstyle{definition}

\newtheorem{exam}[thm]{Example}

\title{Mixed volumes of hypersimplices}
\author{Gaku Liu}

\begin{document}
\maketitle

\begin{abstract}
In this paper we consider mixed volumes of combinations of hypersimplices. These numbers, called ``mixed Eulerian numbers'', were first considered by A.\ Postnikov and were shown to satisfy many properties related to Eulerian numbers, Catalan numbers, binomial coefficients, etc. We give a general combinatorial interpretation for mixed Eulerian numbers and prove the above properties combinatorially. In particular, we show that each mixed Eulerian number enumerates a certain set of permutations in $S_n$. We also prove several new properties of mixed Eulerian numbers using our methods. Finally, we consider a type $B$ analogue of mixed Eulerian numbers and give an analogous combinatorial interpretation for these numbers.
\end{abstract}

\section{Introduction}

For integers $1 \le k \le n$, the \emph{hypersimplex} $\Delta_{k,n} \subset \bb R^{n+1}$ is the convex hull of all points of the form
\[
e_{i_1} + e_{i_2} + \dotsb + e_{i_k}
\]
where $1 \le i_1 < i_2 \dotsb < i_k \le n+1$ and $e_i$ is the $i$-th standard basis vector. Thus, $\Delta_{k,n}$ is an $n$-dimensional polytope which lies in the hyperplane $x_1 + \dotsb + x_{n+1} = k$. Given a polytope $P \subset \bb R^{n+1}$ which lies in a hyperplane $x_1 + \dotsb + x_{n+1} = \alpha$ for some $\alpha \in \bb R$, we define its (normalized) volume $\Vol P$ to be the usual $n$-dimensional volume of the projection of $P$ onto the first $n$ coordinates. 
It is a classical result (usually attributed to Laplace \cite{Lap}) that
\[
n! \Vol \Delta_{k,n} = A(n,k),
\]
where the \emph{Eulerian number} $A(n,k)$ is the number of permutations on $n$ letters with exactly $k-1$ descents.

We now define the mixed volume of a set of polytopes. Given a polytope $P$ and a real number $\lambda \ge 0$, let $\lambda P = \{\lambda x \where x \in P\}$. Given polytopes $P_1$, \dots, $P_m \subset \bb R^n$, let their \emph{Minkowski sum} be
\[
P_1 + \dotsb + P_m = \{ x_1 + \dotsb + x_m \where x_i \in P_i \text{ for all } i \}.
\]
For nonnegative real numbers $\lambda_1$, \dots, $\lambda_m$, the function
\[
f(\lambda_1,\dotsc,\lambda_m) = \Vol(\lambda_1P_1 + \dotsb + \lambda_mP_m)
\]
is known to be a homogeneous polynomial of degree $n$ in the variables $\lambda_1$, \dots, $\lambda_m$. Hence there is a unique symmetric function $\Vol$ defined on $n$-tuples of polytopes in $\bb R^n$ such that
\[
f(\lambda_1,\dotsc,\lambda_m) = \sum_{i_1,\dotsc,i_n = 1}^m \Vol(P_{i_1},\dotsc,P_{i_n})\lambda_{i_1} \dotsb \lambda_{i_n}.
\]
The number $\Vol(P_1,\dotsc,P_n)$ is called the \emph{mixed volume} of $P_1$, \dots, $P_n$.  Mixed volumes of lattice polytopes have important connections to algebraic geometry, where they count the number of solutions to generic systems of polynomial equations; see \cite{Ber}.
If $P_1 = \dotsb = P_n = P$, then $\Vol(P_1,\dotsc,P_n)$ equals the ordinary volume $\Vol(P)$. If $P_1$, \dots, $P_m \subset \bb R^{n+1}$ and each $P_i$ lies in a hyperplane $x_1 + \dotsb + x_{n+1} = \alpha_i$ for some $\alpha_i \in \bb R$, then we define the mixed volume $\Vol(P_1,\dotsc,P_n)$ in terms of the normalized volume defined previously.

Let $c_1$, $c_2$, \dots, $c_n$ be nonnegative integers such that $c_1 + \dotsb+ c_n = n$. We define
\[
A_{c_1,\dotsc,c_n} = n! \Vol(\Delta_{1,n}^{c_1}, \Delta_{2,n}^{c_2}, \dotsc, \Delta_{n,n}^{c_n})
\]
where $(\Delta_{1,n}^{c_1}, \Delta_{2,n}^{c_2}, \dotsc, \Delta_{n,n}^{c_n})$ denotes the $n$-tuple with $c_1$ entries $\Delta_{1,n}$, $c_2$ entries $\Delta_{2,n}$, and so on. The numbers $A_{c_1,\dotsc,c_n}$ are called \emph{mixed Eulerian numbers}, and were introduced by Postnikov in \cite{Post}.

As with ordinary volumes of hypersimplices, mixed volumes of hypersimplices appear to satisfy certain combinatorial identities.
It is immediate that $A_{0^{k-1},n,0^{n-k}} = A(n,k)$, where $0^l$ denotes $l$ entries 0.
Furthermore, the result of Ehrenborg, Readdy, and Steingr\'imsson \cite{ERS} states that
\[
A_{0^{k-2},r,n-r,0^{n-k}}
\]
equals the number of permutations $w \in S_{n+1}$ with $k-1$ descents and $w_1 = r+1$. Other properties are listed in Theorem~\ref{properties} and include
\[
A_{1,\dotsc,1} = n! \quad
A_{k,0,\dotsc,0,n-k} = \binom{n}{k} \quad
A_{c_1,\dotsc,c_n} = 1^{c_1} 2^{c_2} \dotsb n^{c_n} \text{ if } c_1 + \dotsb + c_i \ge i \text{ for all } i. 
\]
These results were proven in \cite{Post} using algebraic and geometric methods. Additional formulas involving mixed Eulerian numbers and their generalizations to other root systems were derived by Croitoru in \cite{Cro}.

In this paper, the main result is a general combinatorial interpretation for the mixed Eulerian numbers which encompasses the previous results. In particular, we show that each mixed Eulerian number enumerates a certain well-defined set of permutations in $S_n$. (When $c_k = n$ and $c_i = 0$ for all $i \neq k$, this set of permutations is precisely the set of permutations with $k-1$ descents.) We show how the above results arise from this result. We also derive some new identities which follow from this interpretation. For example, we show that $A_{c_1,\dotsc,c_n} \le 1^{c_1}2^{c_2} \dotsb n^{c_n}$ for every mixed Eulerian number. We also show that
\begin{equation} \label{introidentity}
A_{n-m, 0^{k-3}, r, m-r, 0^{n-k}} = \sum_{i=0}^{n-m} \binom{m+i}{m} A(m,k-i;r)
\end{equation}
where $A(n,k;r)$ equals the number of permutations $w \in S_{n+1}$ with $k-1$ descents and $w_1 = r+1$. This generalizes the result of Ehrenborg, Readdy, and Steingr\'{i}msson.
The left hand side of \eqref{introidentity} with $r=0$ also appeared in the work of Micha\l ek et.\ al.\ \cite{MSUZ} during their study of exponential families arising from elementary symmetric polynomials. The authors used the recursions of \cite{Cro} to obtain the formula
\[
A_{n-m,0^{k-2},m,0^{n-k}} = \begin{dcases*}
\sum_{i=0}^{n-k} (n-k+1-i)\binom{n-i}{n-m}k^i A(m-i-1,m-n+k-1) & if $n-m < k-1$ \\
k^m & otherwise.
\end{dcases*}
\]

As a secondary result, we define the polytope $\Gamma_{k,n} \subset \bb R^n$ to be the convex hull of all points of the form 
\[
\pm e_{i_1} \pm e_{i_2} \pm \dotsb \pm e_{i_k}
\]
where $1 \le i_1 < \dotsb < i_k \le n$. For nonnegative integers $c_1$, \dots, $c_n$ such that $c_1 + \dotsb + c_n = n$, define
\[
B_{c_1,\dotsc,c_n} = n! \Vol(\Gamma_{1,n}^{c_1}, \Gamma_{2,n}^{c_2}, \dotsc, \Gamma_{n,n}^{c_n}).
\]
We refer to the $B_{c_1,\dotsc,c_n}$ as the \emph{type $B$ mixed Eulerian numbers}, whereas the $A_{c_1,\dotsc,c_n}$ are \emph{type $A$ mixed Eulerian numbers}. We give a combinatorial interpretation for the $B_{c_1,\dotsc,c_n}$ analogous to that of the $A_{c_1,\dotsc,c_n}$ and list several identities that follow from this interpretation.

\section{Permutohedra and signed permutohedra}

We first introduce two polytopes which will be used later in our proofs. Let $y_1$, \dots, $y_{n+1}$ be real numbers. The \emph{permutohedron} $P(y_1,\dotsc,y_n)$ is the convex hull of the $(n+1)!$ points of the form $(y_{w(1)}, \dots, y_{w(n+1)})$, where $w \in S_{n+1}$ is a permutation. For example, $\Delta_{k,n} = P(1^k, 0^{n+1-k})$. The permutohedron is an $n$-dimensional polytope lying in the hyperplane $x_1 + \dotsb + x_{n+1} = y_1 + \dotsb + y_{n+1}$.

We have the following characterizations of $P(y_1,\dotsc,y_{n+1})$; see, for example, \cite{Post}.

\begin{prop} \label{pchar1}
Let $y_1 \ge \dotsb \ge y_{n+1}$ be real numbers. Then $P(y_1,\dotsc,y_n)$ is the set of points $(x_1,\dotsc,x_{n+1}) \in \bb R^{n+1}$ such that for all $1 \le k \le n$ and all $k$-element subsets $\{i_1,\dotsc,i_k\} \subset \{1,\dotsc,n+1\}$, we have
\[
x_{i_1} + \dotsb + x_{i_k} \le y_1 + \dotsb + y_k,
\]
and
\[
x_1 + \dotsb + x_{n+1} = y_1 + \dotsb + y_{n+1}.
\]
\end{prop}

\begin{prop} \label{pchar2}
For nonnegative real numbers $\lambda_1$, \dots, $\lambda_n$, we have
\[
\lambda_1 \Delta_{1,n} + \lambda_2 \Delta_{2,n} + \dotsb + \lambda_n \Delta_{n,n} = P(\lambda_1 + \dotsb + \lambda_n, \lambda_2 + \dotsb + \lambda_n, \dotsc,\lambda_n,0).
\]
Alternatively, if $y_1 \ge \dotsb \ge y_{n+1}$ are real numbers, then $P(y_1,\dotsc,y_{n+1})$ is a translation by $(y_{n+1},\dotsc,y_{n+1})$ of
\[
(y_1-y_2)\Delta_{1,n} + (y_2-y_3)\Delta_{2,n} + \dotsb + (y_n-y_{n+1})\Delta_{n,n}.
\]
\end{prop}

Now let $y_1$, \dots, $y_n$ be real numbers, and define the \emph{signed permutohedron} $SP(y_1,\dotsc,y_n)$ to be the convex hull of the $2^n n!$ points of the form $(\pm y_{w(1)}, \dotsc, \pm y_{w(n)})$, where $w \in S_n$ is a permutation. For example, $\Gamma_{k,n} = SP(1^k,0^{n-k})$. The signed permutohedron is an $n$-dimensional polytope lying in $\bb R^n$.

We have the following characterizations of $SP(y_1,\dotsc,y_n)$.

\begin{prop} \label{spchar1}
Let $y_1 \ge \dotsb \ge y_n \ge 0$ be real numbers. Then $SP(y_1,\dotsc,y_n)$ is the set of points $(x_1,\dotsc,x_n) \in \bb R^n$ such that for all $1 \le k \le n$ and all $k$-element subsets $\{i_1,\dotsc,i_k\} \subset \{1,\dotsc,n\}$, we have
\[
\abs{x_{i_1}} + \dotsb + \abs{x_{i_k}} \le y_1 + \dotsb + y_k.
\]
\end{prop}

\begin{prop} \label{spchar2}
For nonnegative real numbers $\lambda_1$, \dots, $\lambda_n$, we have
\[
\lambda_1 \Gamma_{1,n} + \lambda_2 \Gamma_{2,n} + \dotsb + \lambda_n \Gamma_{n,n} = SP(\lambda_1+\dotsb+\lambda_n, \lambda_2+\dotsb+\lambda_n, \dotsc, \lambda_n).
\]
Alternatively, for real numbers $y_1 \ge \dotsb \ge y_n \ge 0$, we have
\[
SP(y_1,\dotsc,y_n) = (y_1-y_2)\Gamma_{1,n} + (y_2-y_3)\Gamma_{2,n} + \dotsb + (y_{n-1}-y_n)\Gamma_{n-1,n} + y_n \Gamma_{n,n}.
\]
\end{prop}

\section{The main theorem}

\subsection{$C$-permutations}

Let $n$ be a positive integer, and let $S$ be a totally ordered set with $\abs{S} = n$. Let $C = (C_1,\dotsc,C_n)$ be a sequence of $n$ pairwise disjoint sets such that
\begin{itemize}
\item $C_1 \cup \dotsb \cup C_n = S$, and
\item $s < t$ whenever $s \in C_i$, $t \in C_j$, and $i < j$.
\end{itemize}
We will call such a $C$ a \emph{division} of $S$. Let $\abs{C}$ denote the sequence $(\abs{C_1},\dotsc,\abs{C_n})$.

We say that an element $s \in S$ is \emph{admissible} with respect to $C$ if either $s$ is the smallest element of $C_1$, $s$ is the largest element of $C_n$, or $s \in C_i$ for $i \neq 1$, $n$. Given an admissible element $s$, we define the \emph{deletion} of $s$ from $C$ as follows. Let $i$ be such that $s \in C_i$, and let $C_i^- = \{t \in C_i \where t < s\}$ and $C_i^+ = \{t \in C_i \where t > s\}$. The deletion of admissible $s$ from $C$ results in a sequence of $n-1$ sets, denoted by $C^s = (C_1^s,\dotsc,C_{n-1}^s)$, given as follows:
\begin{itemize}
\item If $i = 1$, then $C^s = (C_1^+ \cup C_2, C_3, \dotsc, C_n)$.
\item If $i \neq 1$, $n$, then $C^s = (C_1, \dotsc, C_{i-2}, C_{i-1} \cup C_i^-, C_i^+ \cup C_{i+1}, C_{i+2}, \dotsc, C_n)$.
\item If $i = n$, then $C^s = (C_1, \dotsc, C_{n-2}, C_{n-1} \cup C_n^-)$.
\end{itemize}
In any case, $C^s$ is a division of $S \minus \{s\}$.

Suppose $s_1 \in S$ is admissible with respect to $C$, $s_2 \in S \minus \{s_1\}$ is admissible with respect to $C^{s_1}$, $s_3 \in S \minus \{s_1,s_2\}$ is admissible with respect to $(C^{s_1})^{s_2}$, and so on until $s_i$. Then we say that the sequence $s_1s_2 \dotsc s_i$ is admissible with respect to $C$ and write $((C^{s_1})^{s_2} \dotsc )^{s_i} = C^{s_1 \dotsb s_i}$. If a permutation $s_1 \dotsc s_n$ of $S$ is admissible with respect to $C$, then we call it a \emph{$C$-permutation}.
Note that the number of $C$-permutations depends only on $\abs{C}$.

\begin{exam} \label{example}
Suppose $n=5$ and $C = (\{1\},\emptyset,\{2,3\},\{4\},\{5\})$. The element 2 is admissible with respect to $C$, and $C^2 = (\{1\},\emptyset,\{3,4\},\{5\})$. The element 3 is admissible with respect to $C^2$, and $C^{23} = (\{1\},\emptyset,\{4,5\})$. The element 1 is admissible with respect to $C^{23}$, and $C^{231} = (\emptyset,\{4,5\})$. The element 5 is admissible with respect to $C^{231}$, and $C^{2315} = (\{4\})$. The element 4 is admissible with respect to $C^{2315}$. Hence 23154 is a $C$-permutation. The construction of this permutation is visualized below.

\begin{center}
\begin{tabular}{c}
1 \qquad $\emptyset$ \qquad \textbf{2}3 \qquad 4 \qquad 5 \\
1 \qquad $\emptyset$ \qquad \textbf{3}4 \qquad 5 \\
\textbf{1} \qquad $\emptyset$ \qquad 45 \\
$\emptyset$ \qquad 4\textbf{5} \\
\textbf{4}
\end{tabular}
\end{center}

On the other hand, 23145 is not a $C$-permutation because 4 is not admissible with respect to $C^{231} = (\emptyset,\{4,5\})$.
\end{exam}

\begin{exam} \label{examone}
Suppose $C = (\{1,\dotsc,n\},\emptyset,\dotsc,\emptyset)$. The only element admissible with respect to $C$ is 1, and $C^1 = (\{2,\dotsc,n\},\emptyset,\dotsc,\emptyset)$. The only element admissible with respect to $C^1$ is 2, and so on. Thus the only $C$-permutation is $12 \dotsc n$.

Similarly, if $C = (\emptyset,\dotsc,\emptyset,\{1,\dotsc,n\})$, then the only $C$-permutation is $n(n-1)\dotsc 1$.
\end{exam}

\begin{exam} \label{examall}
Suppose $C$ is a division of $S$ and $\abs{C} = (1,\dotsc,1)$. Then every element of $S$ is admissible with respect to $C$. Moreover, for any element $s \in S$, $C^s$ satisfies $\abs{C^s} = (1,\dotsc,1)$. So by induction, every permutation of $S$ is a $C$-permutation.
\end{exam}

\begin{exam} \label{exambinom}
Let $C$ be a division of the form $C = (C_1, \emptyset, \dotsc, \emptyset, C_n)$. Then the only admissible elements with respect to $C$ are the first element of $C_1$ and the last element of $C_n$. Furthermore, when we delete either of these elements, the resulting sequence of sets is again of the form $(C_1', \emptyset, \dotsc, \emptyset, C_{n-1}')$. So when we construct a $C$-permutation by successively deleting admissible elements, at each step we delete either the first element of the first set or the last element of the last set. Thus the $C$-permutations are the permutations where the  elements of $C_1$ appear in ascending order and the elements of $C_n$ appear in descending order.
\end{exam}

\begin{exam}
We will see from Corollary~\ref{eulerian} that if $1 \le k \le n$ and $C = (\emptyset^{k-1}, \{1,\dotsc,n\}, \emptyset^{n-k})$, then a permutation $w \in S_n$ is a $C$-permutation if and only if it has $k-1$ descents.
\end{exam}

We now state our main result.

\begin{thm} \label{main}
The number of $C$-permutations is $A_{\abs{C}}$.
\end{thm}

\begin{proof}
Let
\begin{align*}
f_n(\lambda_1,\dotsc,\lambda_n) &= \Vol(\lambda_1\Delta_{1,n} + \lambda_2\Delta_{2,n} + \dotsb + \lambda_n\Delta_{n,n}) \\
&= \sum_{c_1+\dotsb+c_n=n} \frac{1}{c_1! \dotsb c_n!} A_{c_1,\dotsc,c_n} \lambda_1^{c_1} \dotsb \lambda_n^{c_n}
\end{align*}
so that
\[
A_{c_1,\dotsc,c_n} = \partial_1^{c_1} \dotsb \partial_n^{c_n} f_n.
\]
The idea of the proof is to write a recursive formula for $f_n$. To do this, we make the following observation:

\begin{prop} \label{propcross}
Let $y_1 \ge \dotsb \ge y_{n+1}$ be real numbers, and let $P = P(y_1,\dotsc,y_{n+1})$. Fix a real number $y_{n+1} \le x \le y_1$, and let $P_x$ denote the cross section of $P$ with the first coordinate equal to $x$. Let $1 \le i \le n$ be such that $y_{i+1} \le x \le y_i$. Then $P_x$ is equal to
\[
\{x\} \times P(y_1,\dotsc,y_{i-1},y_i+y_{i+1}-x,y_{i+2},\dotsc,y_{n+1}).
\]
\end{prop}

\begin{proof}
By Proposition~\ref{pchar1}, $P_x$ is the set of points $(x,x_2,\dotsc,x_{n+1}) \in \bb R^{n+1}$ such that for all $1 \le k \le n-1$ and $k$-element subsets $\{i_1,\dotsc,i_k\} \subset \{2,\dotsc,n+1\}$, we have
\[
x_{i_1} + \dotsb + x_{i_k} \le \min(y_1+\dotsb+y_k, y_1+\dotsb+y_{k+1}-x)
\]
and
\[
x_2 + \dotsb + x_{n+1} = y_1 + \dotsb + y_{n+1}-x.
\]
We have $y_1 + \dotsb + y_k \le y_1 + \dotsb + y_{k+1} - x$ if and only if $x \le y_{k+1}$. Hence, $P_x$ is the set of points $(x,x_2,\dotsc,x_{n+1}) \in \bb R^{n+1}$ such that for all $1 \le k \le n-1$ and $k$-element subsets $\{i_1,\dotsc,i_k\} \subset \{2,\dotsc,n+1\}$, we have
\begin{align*}
x_{i_1} + \dotsb + x_{i_k} &\le y_1 + \dotsb + y_k & \text{ if } x \le y_{k+1} \\
x_{i_1} + \dotsb + x_{i_k} &\le y_1 + \dotsb + y_{k+1} - x & \text{ if } x \ge y_{k+1}
\end{align*}
and
\[
x_2 + \dotsb + x_{n+1} = y_1 + \dotsb + y_{n+1} - x.
\]
By Proposition~\ref{pchar1}, this is precisely the description of
\[
\{x\} \times P(y_1,\dotsc,y_{i-1},y_i+y_{i+1}-x,y_{i+2},\dotsc,y_{n+1}),
\]
as desired.
\end{proof}

\begin{cor} \label{corcross}
Let $\lambda_1$, \dots, $\lambda_n$ be nonnegative real numbers. Fix a real number $0 \le x \le \lambda_1 + \dotsb + \lambda_n$, and let $1 \le i \le n$ be such that $\lambda_{i+1} + \dotsb + \lambda_n \le x \le \lambda_i + \dotsb + \lambda_n$ (where $0 \le x \le \lambda_n$ if $i = n$). Set $t = \lambda_i + \dotsb + \lambda_n - x$. Then the cross section of
\[
\lambda_1\Delta_{1,n} + \lambda_2\Delta_{2,n} + \dotsb + \lambda_n\Delta_{n,n}
\]
with first coordinate equal to $x$ is equal to $\{x\} \times Q$, where $Q$ is the following polytope in the following cases:
\begin{itemize}
\item If $i=1$,
\[
(t+\lambda_2)\Delta_{1,n-1} + \lambda_3\Delta_{2,n-1} + \dotsb + \lambda_n\Delta_{n-1,n-1}.
\]
\item If $2 \le i \le n-1$,
\begin{multline*}
\lambda_1\Delta_{1,n-1} + \dotsb + \lambda_{i-2}\Delta_{i-2,n-1} + (\lambda_{i-1} + \lambda_i - t)\Delta_{i-1,n-1} \\
+ (t + \lambda_{i+1})\Delta_{i,n-1} + \lambda_{i+2}\Delta_{i+1,n-1} + \dotsb + \lambda_n\Delta_{n-1,n-1}.
\end{multline*}
\item If $i=n$,
\[
\lambda_1\Delta_{1,n-1} + \dotsb + \lambda_{n-2}\Delta_{n-2,n-1} + (\lambda_{n-1}+\lambda_n-t)\Delta_{n-1,n-1}.
\]
\end{itemize}
\end{cor}

\begin{proof}
This follows by translating Proposition~\ref{propcross} through Proposition~\ref{pchar2}.
\end{proof}

Corollary~\ref{corcross} now gives the following formula for $f_n$:
\begin{prop} \label{intrecursion}
We have
\begin{align*}
f_n(\lambda_1,\dotsc,\lambda_n) ={ }& \int_0^{\lambda_1} f_{n-1}(t+\lambda_2,\lambda_3,\dotsc,\lambda_n)\,dt \\
& + \sum_{i = 2}^{n-1} \int_0^{\lambda_i} f_{n-1}(\lambda_1, \dotsc, \lambda_{i-2}, \lambda_{i-1}+\lambda_i-t, t+\lambda_{i+1}, \lambda_{i+2}, \dotsc, \lambda_n)\,dt \\
& + \int_0^{\lambda_n} f_{n-1}(\lambda_1,\dotsc,\lambda_{n-2},\lambda_{n-1}+\lambda_n-t)\,dt.
\end{align*}
\end{prop}

Now, we wish to use this formula to calculate $\partial_1^{c_1} \dotsb \partial_n^{c_n} f_n$. We use the ``differentiation under the integral'' rule: For smooth functions $f(x)$ and $g(x,t)$, we have
\[
\frac{d}{dx} \int_0^{f(x)} g(x,t)\, dt = f'(x)g(f(x),t) + \int_0^{f(x)} \frac{\partial}{\partial x} g(x,t)\, dt.
\]
It follows that for $2 \le i \le n-1$, we have
\begin{multline} \label{differentiate}
\left( \frac{\partial}{\partial \lambda_i} \right)^{c_i} \int_0^{\lambda_i} f_{n-1}(\lambda_1,\dotsc,\lambda_{i-1}+\lambda_i-t,t+\lambda_{i+1},\dotsc,\lambda_n)\, dt \\
= \sum_{r=0}^{c_i-1} \partial_{i-1}^r \partial_i^{c_i-r-1} f_{n-1}(\lambda_1,\dotsc,\lambda_i+\lambda_{i+1},\dotsc,\lambda_n) \\
+ \int_0^{\lambda_i} \partial_{i-1}^{c_i} f_{n-1}(\lambda_1,\dotsc,\lambda_{i-1}+\lambda_i-t,t+\lambda_{i+1},\dotsc,\lambda_n)\, dt
\end{multline}
and hence
\begin{align*}
\left( \frac{\partial}{\partial \lambda_1} \right)^{c_1} \dotsb \left( \frac{\partial}{\partial \lambda_n} \right)^{c_n} \int_0^{\lambda_i} f_{n-1}(\lambda_1,\dotsc&\,,\lambda_{i-1}+\lambda_i-t,t+\lambda_{i+1},\dotsc,\lambda_n)\, dt \\
&= \sum_{r=0}^{c_i-1} \partial_1^{c_1} \dotsb \partial_{i-1}^{c_{i-1}} \partial_{i-1}^r \partial_i^{c_i-r-1} \partial_i^{c_{i+1}} \dotsb \partial_{n-1}^{c_n} f_{n-1} \\
&= \sum_{r=0}^{c_i-1} A_{c_1, \dotsc, c_{i-2}, c_{i-1}+r, c_i-r-1+c_{i+1}, c_{i+2}, \dotsc, c_n} \\
&= \sum_{s \in C_i} A_{\abs{C^s}}
\end{align*}
where $C$ is a division with $\abs{C} = (c_1,\dotsc,c_n)$. Note that the final term of \eqref{differentiate} vanishes after differentiation because $f_{n-1}$ is a polynomial of degree $n-1$.

By similar (and simpler) calculations, we have
\begin{align*}
\left( \frac{\partial}{\partial \lambda_1} \right)^{c_1} \dotsb \left( \frac{\partial}{\partial \lambda_n} \right)^{c_n} \int_0^{\lambda_1} f_{n-1}(t+\lambda_2,\lambda_3,\dotsc,\lambda_n)\, dt &= A_{c_1+c_2-1,c_3,\dotsc,c_n} \\
&= A_{\abs{C^1}}
\end{align*}
and
\begin{align*}
\left( \frac{\partial}{\partial \lambda_1} \right)^{c_1} \dotsb \left( \frac{\partial}{\partial \lambda_n} \right)^{c_n} \int_0^{\lambda_n} f_{n-1}(\lambda_1,\dotsc,\lambda_{n-2},\lambda_{n-1}+\lambda_n-t)\, dt &= A_{c_1,\dotsc,c_{n-1}+c_n-1} \\
&= A_{\abs{C^n}}.
\end{align*}
Combining these calculations with Proposition~\ref{intrecursion}, we obtain
\[
A_{c_1,\dotsc,c_n} = A_{\abs{C^1}} + \sum_{i=2}^{n-1}\sum_{s \in C_i} A_{\abs{C^s}} + A_{\abs{C^n}}.
\]
The desired result now follows by induction with the base case $A_1 = 1$.
\end{proof}

While $C$-permutations are defined recursively in general, there are certain cases where more explicit descriptions can be given. This allows us to derive various formulas for mixed Eulerian numbers, which we do in Section~\ref{sec:properties}.

\subsection{Index functions and superdiagonality}

We will also associate each $C$-permutation with a function which we call an ``index function''. For some applications, this function will be more useful to work with than the permutation itself. This section will only be used for Section~\ref{proofineq} and can be returned to later.

Let $C = (C_1,\dotsc,C_n)$ be a division of $S$ and let $w = w_1 \dotsc w_n$ be a $C$-permutation. For each $1 \le i \le n$, the \emph{index} of $w_i$ in $w$ with respect to $C$ is the $j$ such that $w_i \in C^{w_1w_2\dotsc w_{i-1}}_j$. In other words, $j$ is the index of the set containing $w_i$ immediately before we delete $w_i$. Let $I^C_w : S \to \bb N$ be the function which takes each $s \in S$ to its index in $w$ with respect to $C$. Note that if $s \in C_i$, then $I^C_w(s) \in \{1,\dotsc,i\}$. We will call \emph{any} function $I : S \to \bb N$ which maps $C_i$ into $\{1,\dotsc,i\}$ and \emph{index function} of $C$.

\begin{exam}
Let $C = (\{1\},\emptyset,\{2,3\},\{4\},\{5\})$ and $w = 23154$ as in Example~\ref{example}. Then $I^C_w(2) = 3$, $I^C_w(3) = 3$, $I^C_w(1) = 1$, $I^C_w(5) = 2$, and $I^C_w(4) = 1$.
\end{exam}

\begin{exam} \label{indexeulerian}
Let $C = (\emptyset^{k-1}, \{1,\dotsc,n\}, \emptyset^{n-k})$ and let $w$ be a $C$-permutation. By Corollary~\ref{eulerian}, we can uniquely write $w = \overline{w_1} \ \overline{w_2} \dotsc \overline{w_{k}}$ where each $\overline{w_i}$ is an increasing sequence and $w$ is the concatenation of these sequences. Then by Proposition~\ref{descents}, if $s$ is a term in $\overline{w_i}$, then $I_w^C(s) = k-i+1$.
\end{exam}

We introduce some final terminology. Call a division $C$ \emph{superdiagonal} if $\abs{C_1} + \dotsb + \abs{C_i} \ge i$ for all $i$. Call a division \emph{subdiagonal} if $\abs{C_n} + \abs{C_{n-1}} + \dotsb + \abs{C_{n-i+1}} \ge i$ for all $i$. We make the following observation, which is easy to check.

\begin{prop} \label{superdiagonal}
If $C$ is a superdiagonal (resp., subdiagonal) division of $S$, then for any admissible $s \in S$, $C^s$ is also superdiagonl (resp., subdiagonal).
\end{prop}

The following is the main result on index functions, which we prove in the next section.

\begin{prop} \label{injection}
Let $C = (C_1,\dotsc,C_n)$ be a division of $S$. Then the map $w \mapsto I^C_w$ is an injection from the set of $C$-permutations to the set of index functions of $C$. This map is a bijection if and only if $C$ is superdiagonal.
\end{prop}

\section{Properties of mixed Eulerian numbers} \label{sec:properties}

Using algebraic and geometric techniques, Postnikov proved the following properties of mixed Eulerian numbers.

\begin{thm}[Postnikov \cite{Post}] \label{properties}
The mixed Eulerian numbers have the following properties:
\renewcommand{\theenumi}{\alph{enumi}}
\renewcommand{\labelenumi}{(\theenumi)}
\begin{enumerate}
\item\label{propintegers} The numbers $A_{c_1,\dotsc,c_n}$ are positive integers defined for $c_1$, \dots, $c_n \ge 0$, $c_1 + \dotsb + c_n = n$.
\item\label{propreverse} We have $A_{c_1,\dotsc,c_n} = A_{c_n,\dotsc,c_1}$.
\item\label{propeulerian} For $1 \le k \le n$, the number $A_{0^{k-1},n,0^{n-k}}$ equals the usual Eulerian number $A(n,k)$. Here, $0^l$ denotes a sequence of $l$ zeroes.
\item\label{propweightedsum} We have $\sum \frac{1}{c_1! \dotsb c_n!} A_{c_1,\dotsc,c_n} = (n+1)^{n-1}$, where the sum is over nonnegative integer sequences $c_1$, \dots, $c_n$ with $c_1 + \dotsb + c_n = n$.
\item\label{propsum} We have $\sum A_{c_1,\dotsc,c_n} = n! C_n$, where the sum is over nonnegative integer sequences $c_1$, \dots, $c_n$ with $c_1 + \dotsb + c_n = n$, and $C_n = \frac{1}{n+1} \binom{2n}{n}$ is the $n$-th Catalan number.
\item\label{propadjacent} For $2 \le k \le n$ and $0 \le r \le n$, the number $A_{0^{k-2},r,n-r,0^{n-k}}$ is equal to the number of permutations $w \in S_{n+1}$ with $k-1$ descents and $w_1 = r+1$.
\item\label{propall} We have $A_{1,\dotsc,1} = n!$.
\item\label{propbinom} We have $A_{k,0,\dotsc,0,n-k} = \binom{n}{k}$.
\item\label{propeq} We have $A_{c_1,\dotsc,c_n} = 1^{c_1}2^{c_2} \dotsb n^{c_n}$ is $c_1 + \dotsb + c_i \ge i$ for all $i$.
\end{enumerate}
\end{thm}

\begin{thm}[Postnikov \cite{Post}] \label{cycle}
Let $\sim$ denote the equivalence relation on the set of nonnegative integer sequences $(c_1,\dotsc,c_n)$ with $c_1 + \dotsb + c_n = n$ given by $(c_1,\dotsc,c_n) \sim (c_1',\dotsc,c_n')$ whenever $(c_1,\dotsc,c_n,0)$ is a cyclic shift of $(c_1',\dotsc,c_n',0)$. Then for a fixed $(c_1,\dotsc,c_n)$, we have
\[
\sum_{(c_1',\dotsc,c_n') \sim (c_1,\dotsc,c_n)} A_{c_1',\dotsc,c_n'} = n!.
\]
\emph{Note:} There are exactly $C_n = \frac{1}{n+1}\binom{2n}{n}$ equivalence classes.
\end{thm}

We now prove how these properties arise from the combinatorial interpretation of mixed Eulerian numbers given by Theorem~\ref{main}. We also give the following three additional properties.

\begin{thm} \label{ineq}
We have $A_{c_1,\dotsc,c_n} \le 1^{c_1}2^{c_2} \dotsb n^{c_n}$, with equality if and only if $c_1 + \dotsb + c_i \ge i$ for all $i$.
\end{thm}

\begin{thm} \label{hybrid}
Let $c_1$, \dots, $c_n$ be nonnegative integers such that $c_1 + \dotsb + c_n = n$, and suppose there exists some $0 \le r \le n$ such that $c_1 + \dotsb + c_i \ge i$ for all $1 \le i \le r$ and $c_n + c_{n-1} + \dotsb + c_{n-i+1} \ge i$ for all $1 \le i \le n-r$. Then
\[
A_{c_1,\dotsc,c_n} = \binom{n}{c_1+\dotsb+c_r} 1^{c_1} 2^{c_2} \dotsb r^{c_r} 1^{c_n} 2^{c_{n-1}} \dotsb (n-r)^{c_{r+1}}.
\]
\end{thm}



\begin{thm} \label{partialeulerian}
We have
\[
A_{n-m, 0^{k-3}, r, m-r, 0^{n-k}} = \sum_{i=0}^{n-m} \binom{m+i}{m} A(m,k-i;r)
\]
where $A(n,k;r)$ equals the number of permutations $w \in S_{n+1}$ with $k-1$ descents and $w_1 = r+1$. In particular,
\[
A_{n-m,0^{k-2},m,0^{n-k}} = \sum_{i=0}^{n-m} \binom{m+i}{m} A(m,k-i)
\]
where $A(n,k)$ is defined to be 0 if $k \le 0$ or $k > n$.
\end{thm}

We do not have a combinatorial proof of Theorem 3.1(\ref{propweightedsum}), which was proven using the volume of the permutohedron.

\subsection{Proofs of Theorem \ref{properties}}

Property~\eqref{propintegers} is clear.

Property~\eqref{propreverse} follows from the fact that if $w$ is a $(C_1,\dotsc,C_n)$-permutation, then $w$ is also a $(C_n,\dotsc,C_1)$-permutation with the reverse ordering on $C_1 \cup \dotsb \cup C_n$.

Property~\eqref{propadjacent}, which is a generalization of property~\eqref{propeulerian}, follows from the following proposition.

\begin{prop} \label{generaleulerian}
Let $2 \le k \le n$ and $0 \le r \le n$. Let $C$ be a division of $S$ with $\abs{C} = (0^{k-2},r,n-r,0^{n-k})$. Let $\lambda$ be an element not in $S$ such that $\lambda > s$ for all $s \in C_{k-1}$ and $\lambda < s$ for all $s \in C_k$. Then a permutation $w = w_1 \dotsc w_n$ of $S$ is a $C$-permutation if and only if the sequence $\lambda$, $w_1$, \dots, $w_n$ has $k-1$ descents.
\end{prop}

\begin{proof}
We induct on $n$. The argument below will work for $n=2$ without assuming the inductive hypothesis, so we will have a base case. Assume without loss of generality that $S = \{1,\dotsc,n\}$. Assume $w = w_1 \dotsc w_n$ is a $C$-permutation. First suppose $w_1 \le r$. If $k > 2$, then $\abs{C^{w_1}} = (0^{k-3}, w_1-1, n-w_1, 0^{n-k})$. Since $w_2 \dotsc w_n$ is a $C^{w_1}$-permutation, the inductive hypothesis then implies that the sequence $w_1$, $w_2$, \dots, $w_n$ has $k-2$ descents. If $k=2$, then since $w_1 \le r$ and $w_1$ is admissible with respect to $C$, we must have $w_1 = 1$ and $\abs{C^{w_1}} = (n-1,0^{n-2})$. Thus $w_2 \dotsc w_n = 2 \dotsc n$ (see Example~\ref{examone}), so $w_1 \dotsc w_n = 1 \dotsc n$. In either case, $w_1$, \dots, $w_n$ has $k-2$ descents. Since $w_1 \le r$, it follows that $\lambda$, $w_1$, \dots, $w_n$ has $k-1$ descents, as desired. The argument for $w_1 > r$ follows analogously, with $k = n$ being the special case instead of $k = 2$.

Conversely, suppose $w = w_1 \dots w_n$ is a permutation of $S$ such that $\lambda$, $w_1$, \dots, $w_n$ has $k-1$ descents. First suppose $w_1 \le r$. Hence $w_1$, $w_2$, \dots, $w_n$ has $k-2$ descents. If $k > 2$, then $w_1$ is admissible with respect to $C$ and $\abs{C^{w_1}} = (0^{k-3}, w_1-1, n-w_1, 0^{n-k})$. The inductive hypothesis then implies that $w_2 \dotsc w_n$ is a $C^{w_1}$-permutation. If $k=2$, then $w_1$, \dots, $w_n$ has no descents, so $w = 1 \dotsc n$. It is easy to see that this is a $C$-permutation. In either case, we have that $w$ is a $C$-permutation. The argument for $w_1 > r$ follows analogously.
\end{proof}

\begin{cor} \label{eulerian}
Let $1 \le k \le n$ and let $C = (\emptyset^{k-1}, \{1,\dotsc,n\}, \emptyset^{n-k})$. Then a permutation $w \in S_n$ is a $C$-permutation if and only if it has $k-1$ descents.
\end{cor}

\begin{proof}
Take $r=0$ or $n$ in the previous Proposition.
\end{proof}

We can also consider descents of ``unfinished'' permutations which are admissible with respect to $C$. The proof is similarly by induction; we omit it here.

\begin{prop} \label{descents}
Let $C$ be a division with $\abs{C} = (0^{k-1},n,0^{n-k})$. Suppose that the sequence $s_1s_2 \dotsc s_i$ is admissible with respect to $C$. Let $j$ be the index such that $s_i \in C^{s_1\dotsc s_{i-1}}_j$. Then $s_1s_2 \dotsc s_i$ has $k-j$ descents.
\end{prop}

Property~\eqref{propsum} follows from Theorem~\ref{cycle}, which is proven in section~\ref{proofcycle}.

Property~\eqref{propall} follows from Example~\ref{examall}.

Property~\eqref{propbinom} follows from Example~\ref{exambinom}.

Property~\eqref{propeq} is implied by Theorem~\ref{ineq}, which we prove in the next section.

\subsection{Proof of Theorem~\ref{ineq}} \label{proofineq}

It suffices to prove Proposition~\ref{injection}.

We first prove injectivity. Let $w = w_1 \dotsc w_n$ be a $C$-permutation, and set $I = I_w^C$. We wish to show that $w$ is determined by $I$. It suffices to show that $w_1$ is determined by $I$. Indeed, if we prove this, then since $w_2 \dotsc w_n$ is a $C^{w_1}$-permutation, the same argument would imply that $w_2$ is determined by $I_{w_2 \dotsc w_n}^{C^{w_1}}$, and this function is determined as the restriction of $I$ to $S \minus \{w_1\}$. The terms $w_3$, $w_4$, are determined similarly.

Let $i_1$ be such that $w_1 \in C_{i_1}$. Then $I(w_1) = i_1$. Let $i$ be the largest number such that there exists some $s \in C_{i}$ with $I(s) = i$, and consider the smallest such $s$. By definition, $i_1 \le i$. If $i_1 < i$, then after we delete $w_1$ from $C$ we have $s \in C_{i-1}^{w_1}$. Hence $I(s) \le i-1$, contradicting the definition of $s$. So $i_1 = i$. Now if $w_1 > s$, then after we delete $w_1$ from $C$ we obtain $s \in C_{i-1}^{w_1}$, again a contradiction. Hence $w_1 = s$. Thus $w_1$ is determined by $I$, as desired.

We now prove surjectivity in the case where $C$ is superdiagonal. We induct on $n$. The case $n=1$ is trivial. Suppose $C$ is superdiagonal. Let $I$ be an index function for $C$. We wish to construct a $C$-permutation $w$ such that $I_w = I$. First note that $\abs{C_1} \ge 1$, and any element $s \in C_1$ satisfies $I(s) = 1$. Thus we can let $i$ be the largest number such that there exists some $s \in C_{i}$ with $I(s) = i$, and we consider the smallest such $s$. Since $\abs{C_n} \le 1$, it follows that $s$ is admissible with respect to $C$. By Proposition~\ref{superdiagonal}, $C^{s}$ is superdiagonal.

Let $I' : S \minus \{s\} \to \bb N$ be the restriction of $I$ to $S \minus \{s\}$. We claim that $I'$ is an index function of $C^{s}$. Indeed, let $s' \in S \minus \{s\}$ and let $i'$ be such that $s' \in C_{i'}^s$. We wish to prove $I'(s') \in \{1,\dotsc,i'\}$. We have either $s' \in C_{i'}$ or $s' \in C_{i'+1}$. In the first case, we are done since $I'(s') = I(s') \in \{1,\dotsc,i'\}$. In the second case, we must have either $i'+1 > i$ or $i'+1 = i$ and $s' < s$. By the definition of $i$ and $s$, we must therefore have $I(s') \neq i'+1$, and hence $I'(s') = I(s') \in \{1,\dotsc,i'\}$, as desired. Thus $I'$ is an index function for $C^{s}$.

Since $C^{s}$ is superdiagonal and $I'$ is an index function for $C^{s}$, by the inductive hypothesis there exists a $C^{s}$-permutation $w'$ such that $I_{w'}^{C^{s}} = I'$. Letting $w = sw'$, we have that $I_w^C = I$, as desired. This proves surjectivity.

Conversely, suppose $C$ is not superdiagonal. The function $I : S \to \bb N$ with $I(s) = 1$ for all $s \in S$ is clearly an index function of $C$. Suppose there exists a $C$-permutation $w  = s_1 \dotsc, s_n$ with $I_w = I$. Thus when we successively delete $s_1$, \dots, $s_n$ from $C$, we only ever delete from the first set in the current sequence. Hence we only ever delete the smallest remaining element.

Let $i$ be the largest number such that $\abs{C_1} + \dotsb + \abs{C_i} < i$. Hence $i < n$ and $C_{i+1}$ is nonempty. Let $s$ be the smallest element of $C_{i+1}$. After deleting $s_1$, \dots, $s_{\abs{C_1} + \dotsb + \abs{C_i}}$ from $C$, the smallest remaining element is $s$. But $\abs{C_1} + \dotsb + \abs{C_i} < i$, so after the above deletions, $s$ is not in the first set of the sequence. This contradicts $I(s) = 1$. So there is no $w$ such that $I_w = I$, as desired. This proves Proposition~\ref{injection}.

\subsection{Proof of Theorem~\ref{cycle}} \label{proofcycle}

Let $n$ be a positive integer and let $C = (C_1,\dotsc,C_n)$ be a division of $\{1,\dotsc,n\}$ with $\abs{C} = (c_1,\dotsc,c_n)$. Set $C_{n+1} = \emptyset$.

We will describe a process which is a cyclic version of the construction of $C$-permutations. Arrange the numbers 1, \dots, $n$ around a circle $\mc C$ clockwise in that order. We will define $n+1$ ``blocks'' as follows: for each $1 \le i \le n+1$, block $B_i$ initially contains the elements of $C_i$. We view $B_1$, \dots, $B_{n+1}$ as being arranged around $\mc C$ in that order, including the empty blocks; i.e.\ $B_i$ is viewed as being between $B_{i-1}$ and $B_{i+1}$ even if $B_i$ is empty. For any element $s \in \{1,\dotsc,n\}$, we define the deletion of $s$ from $\mc C$ as follows. Suppose $s \in B_i$. Let $B_i^-$ be the set of elements in $B_i$ which are to the left of (counterclockwise from) $s$, and let $B_i^+$ be the set of elements in $B_i$ which are to the right of (clockwise from) $s$. To delete $s$, we remove $s$ \emph{and} the block $B_i$ from $\mc C$, put all the elements of $B_i^-$ into the block to the left of $B_i$, and put all the elements of $B_i^+$ into the block to the right of $B_i$. The order of the undeleted elements remains unchanged. We can then delete another element, and so on. After we delete all $n$ elements, we are left with only one block, which is empty. Since a nonempty block remains nonempty until it is deleted, this final empty block was originally empty and remained so throughout the process.

Let $w = w_1 \dotsc w_n \in S_n$ be a permutation. Let $r(w)$ be the $r$ such that $B_r$ is the final block that remains when we successively delete $w_1$, \dots, $w_n$ from $\mc C$. It is not hard to see that for each $r$ with $C_r = \emptyset$, the set of $w$ such that $r(w) = r$ is precisely the set of $(C_{r+1}, C_{r+2}, \dotsc, C_{r-1})$-permutations, where the indices of the $C_i$ are taken modulo $n+1$ and the elements $\{1,\dotsc,n\}$ are ordered starting from the first element of $C_{r+1}$ and going cyclically to the last element of $C_{r-1}$. There are $A_{c_{r+1},\dotsc,c_{r-1}}$ such permutations. Hence we have
\[
n! = \sum_{c_r=0} A_{c_{r+1},c_{r+2},\dotsc,c_{r-1}}
\]
which is exactly what we wanted to prove.

\subsection{Proof of Theorem~\ref{hybrid}}

Note that the hypotheses on $c_1$, \dots, $c_n$  imply that $c_1 + \dotsb + c_r = r$ and $c_{r+1} + \dotsb + c_n = n-r$. Let $C = (C_1,\dotsc,C_n)$ be a division with $\abs{C} = (c_1,\dotsc,c_n)$. Let $S^- = C_1 \cup \dotsb \cup C_r$ and $S^+ = C_{r+1} \cup \dotsb \cup C_n$. Let $C^- = (C_1,\dotsc,C_r)$ and $C^+ = (C_{r+1},\dotsc,C_n)$. Hence $C^-$ and $C^+$ are divisions of $S^-$ and $S^+$, respectively, and $C^-$ is superdiagonal and $C^+$ is subdiagonal. We write $C = (C^-,C^+)$ to indicate that $C$ is the concatenation of the sequences $C^-$, $C^+$.

Suppose $s \in S^-$ is admissible with respect to $C$. We claim that $s$ is admissible respect to $C^-$ and $C^s = ((C^-)^s, C^+)$. Indeed, this is clearly true if $s \in C_i$ for $i < r$, and it is true if $s \in C_r$ because $\abs{C_r} \le 1$. Similarly, if $s \in S^+$ is admissible with respect to $C$, then $s$ is admissible with respect to $C^+$ and $C^s = (C^-,(C^+)^s)$. Moreover, by Proposition~\ref{superdiagonal}, $C^-$ and $C^+$ remain superdiagonal and subdiagonal, respectively, after deleting elements. Hence, successively deleting elements from $C$ is equivalent to successively deleting elements from $C^-$ and $C^+$. We can thus bijectively construct any $C$-permutation $s_1 \dotsc s_n$ by specifying a $C^-$-permutation, specifying a $C^+$-permutation, and specifying the values of $i$ for which $s_i$ is an element of $S^-$. There are $\binom{n}{c_1 + \dotsb + c_r}$ ways to specify the values of $i$, and by Theorems~\ref{properties}\eqref{propeq} and \eqref{propreverse}, there are $1^{c_1} 2^{c_2} \dotsb r^{c_r}$ $C^-$-permutations and $1^{c_n} 2^{c_n-1} \dotsb (n-r)^{c_{r+1}}$ $C^+$-permutations. This gives the desired result.

\subsection{Proof of Theorem~\ref{partialeulerian}}

We will in fact prove a more general identity. Fix a division $C$ of $S$ such that
\[
\abs{C} = (n-m,0^{k-3},r,m-r,0^{n-k})
\]
for some $0 \le r \le m \le n$ and $3 \le k \le n$. Suppose $s_1$, $s_2$, \dots\ is a sequence of elements, not necessarily all in $S$. We call a term $s_i$ of this sequence a \emph{$C_1$-descent} if either
\begin{itemize}
\item $s_i \in C_1$, or
\item there exists $j > i$ such that $s_j \notin C_1$, $s_i > s_j$, and $s_k \in C_1$ for every $i < k < j$.
\end{itemize}
Note that if $C_1$ is empty, then a $C_1$-descent is just an ordinary descent.

We can now state the result.

\begin{prop} \label{generalpartialeulerian}
Let $C$ be as above, and let $w_0 = \lambda$ be an element not in $S$ such that $\lambda > s$ for all $s \in C_{k-1}$ and $\lambda < s$ for all $s \in C_k$. Then a permutation $w = w_1 \dotsc w_n$ of $S$ is a $C$-permutation if and only if the sequence $w_0$, $w_1$, $w_2$, \dots, $w_n$ satisfies the following properties:
\begin{enumerate}
\renewcommand{\theenumi}{\alph{enumi}}
\renewcommand{\labelenumi}{(\theenumi)}
\item If $i < j$ and $w_i$, $w_j \in C_1$, then $w_i < w_j$. \label{goodascending}
\item The sequence has at least $k-1$ $C_1$-descents. \label{goodatleast}
\item If $w_{i}$ is the $(k-1)$-th $C_1$-descent, then $w_{i+1}$, $w_{i+2}$, \dots, $w_n$ is an increasing sequence. \label{goodend}
\end{enumerate}
\end{prop}

Note that if $C_1 = \emptyset$, this proposition becomes Proposition~\ref{generaleulerian}.

\begin{proof}[Proof of Proposition~\ref{generalpartialeulerian}]
The proof is similar to that of Proposition~\ref{generaleulerian}. We induct on $n$. The below argument will work for $n=3$ without the inductive hypothesis, so we will have a base case. Call a sequence \emph{$(t,T)$-good} if it satisfies properties (a)--(c) with $k$ replaced with $t$ and $C_1$ replaced with $T$. Without loss of generality, assume $C_1 = \{1', 2', \dotsc, (n-m)'\}$ and $C_{k-1} \cup C_k = \{1,2,\dotsc,m\}$, with the obvious ordering on these two sets.

Suppose $w$ is a $C$-permutation. First suppose $w_1 \in C_1$.  Then $w_1 = 1'$. If $k > 3$, then $\abs{C^{w_1}} = (n-m-1, 0^{k-4}, r, m-r,0^{n-k})$. The inductive hypothesis then implies that the sequence $\lambda$, $w_2$, \dots, $w_n$ is $(k-1,C_1 \minus \{1'\})$-good. It then follows that that $\lambda$, $1'$, $w_2$, \dots, $w_n$ is $(k,C_1)$-good, as desired. If $k = 3$, then $\abs{C^{w_1}} = (n-m+r-1, m-r, 0^{n-3})$. By Proposition~\ref{generaleulerian}, it follows that $\lambda$, $w_2$, \dots, $w_n$ has 1 descent in the ordinary sense. It is easy to check that this implies $\lambda$, $1'$, $w_2$, \dots, $w_n$ is $(3,C_1)$-good.

Now suppose $w_1 \in C_{k-1}$. If $k > 3$, then $\abs{C^{w_1}} = (n-m, 0^{k-4}, w_1-1, m-w_1, 0^{n-k})$. By the inductive hypothesis, the sequence $w_1$, $w_2$, \dots, $w_n$ is $(k-1,C_1)$-good. Since $\lambda > w_1$, it follows that $\lambda$, $w_1$, \dots, $w_n$ is $(k,C_1)$-good, as desired. If $k = 3$, then $\abs{C^{w_1}} = (n-m+w_1-1, m-w_1, 0^{n-k})$. Proposition~\ref{generaleulerian} then implies that $w_1$, $w_2$, \dots, $w_n$ has 1 descent in the ordinary sense. It is easy to check that this implies $\lambda$, $w_1$, \dots, $w_n$ is $(3,C_1)$-good.

Finally, suppose $w_1 \in C_k$. If $k < n$, the argument works similarly as in the previous paragraph. Suppose $k = n$. Then $w_1 = m$ and $\abs{C^{w_1}} = (n-m, 0^{n-3}, m-1)$. By Example~\ref{exambinom}, this implies that in the sequence $w_2$, \dots, $w_n$, the elements of $\{1',\dots,(n-m)'\}$ appear in ascending order and the elements $\{1,\dots,m-1\}$ appear in descending order. Since $w_1 = m$, the same can be said of the sequence $w_1$, \dots, $w_n$. This implies that every term except the last term of this sequence is a $C_1$-descent. Thus, the sequence is $(n,C_1)$-good. Since $\lambda < w_1$, the sequence $\lambda$, $w_1$, \dots, $w_n$ is also $(n,C_1)$-good, as desired.

Conversely, suppose $w$ is a permutation of $S$ such that $\lambda$, $w_1$, \dots, $w_n$ is $(k,C_1)$-good. First suppose $w_1 \in C_1$. By \eqref{goodascending}, we must have $w_1 = 1'$. Hence $\lambda$, $w_2$, \dots, $w_n$ is $(k-1,C_1\minus\{1'\})$-good and $w_1$ is $C$-admissible. If $k > 3$, then $\abs{C^{w_1}} = (n-m-1, 0^{k-4}, r, m-r,0^{n-k})$, so by the inductive hypothesis, $w_2 \dotsc w_n$ is a $C^{w_1}$-permutation. If $k=3$, then $\abs{C^{w_1}} = (n-m+r-1, m-r, 0^{n-3})$. Since $\lambda$, $w_2$, \dots, $w_n$ is $(2,C_1\minus\{1'\})$-good, it has exactly 1 descent in the ordinary sense. So by Proposition~\ref{generaleulerian}, $w_2 \dotsc w_n$ is a $C^{w_1}$-permutation. Either way, $w$ is a $C$-permutation, as desired.

Now suppose $w_1 \in C_{k-1}$. Then the sequence $w_1$, \dots, $w_n$ is $(k-1,C_1)$-good, and $w_1$ is $C$-admissible. If $k > 3$, then $\abs{C^{w_1}} = (n-m, 0^{k-4}, w_1-1, m-w_1, 0^{n-k})$. The inductive hypothesis then implies $w_2 \dotsc w_n$ is a $C^{w_1}$-permutation. If $k=3$, then $w_1$, \dots, $w_n$ has exactly one descent in the ordinary sense, and $\abs{C^{w_1}} = (n-m+w_1-1, m-w_1, 0^{n-k})$. Proposition~\ref{generaleulerian} then implies that $w_2\dotsc w_n$ is a $C^{w_1}$-permutation. Either way, $w$ is a $C$-permutation.

Finally, suppose $w_1 \in C_k$. If $k < n$, the argument works similarly as in the previous paragraph. Suppose $k = n$. Then the sequence $\lambda$, $w_1$, \dots, $w_n$ has $n-1$ $C_1$-descents. But $\lambda < w_1$, so in this sequence the terms $w_1$, $w_2$, \dots, $w_{n-1}$ must all be $C_1$-descents. This implies that the elements of $\{1,\dotsc,m\}$ appear in this sequence in descending order. By (a), the elements of $\{1',\dotsc,(n-m)'\}$ appear in ascending order. It is easy to check that this implies $w$ is a $C$-permutation, as desired. 
\end{proof}

We now want a way to enumerate the permutations from Proposition~\ref{generalpartialeulerian}.
Given a set $S$, define a \emph{$\star$-permutation} of $S$ to be a finite sequence $s_1s_2\dotsc$ consisting of elements of $S$ and ``$\star$'' symbols such that every element of $S$ appears exactly once. A \emph{$\star$-descent} of a $\star$-permutation $s_1s_2\dotsc$ is an index $i$ such that either $s_i = \star$ or there exists some $j > i$ with $s_i$, $s_j \in S$, $s_i > s_j$, and $s_k = \star$ for every $i < k < j$.

\begin{prop}
Let $C$ be a division of $S$ with $\abs{C} = (n-m, 0^{k-3}, r, m-r, 0^{n-k})$, and let $\lambda$ be a number such that $\lambda > s$ for all $s \in C_{k-1}$ and $\lambda < s$ for all $s \in C_k$. Then the $C$-permutations are in bijection with $\star$-permutations $s_1s_2\dotsc$ of $C_{k-1} \cup \{\lambda\} \cup C_k$ for which 
\begin{itemize}
\item $s_1 = \lambda$ 
\item The number of $\star$'s is at most $n-m$.
\item The number of $\star$-descents is equal to $k-1$.
\end{itemize}
\end{prop}

\begin{proof}
Suppose $s = s_1s_2\dotsc$ is a $\star$-permutation of $C_{k-1} \cup \{\lambda\} \cup C_k$ satisfying the above conditions. Let $i$ be the $(k-1)$-th $\star$-descent of $s$. We obtain a $C$-permutation from $s$ as follows: Begin with the subsequence $s_2\dotsc s_i$, and replace the first $\star$ with the first element of $C_1$, the second $\star$ with the second element of $C_1$, and so on, until all $\star$'s are replaced. Call the new sequence $w' = w_1\dotsc w_{i-1}$. Append to the end of $w'$ the elements of $S \minus \{w_1,\dotsc,w_{i-1}\}$ in ascending order. The result is a $C$-permutation by Proposition~\ref{generalpartialeulerian}.

Now suppose $w = w_1\dotsc w_n$ is a $C$-permutation. Append $\lambda$ to the beginning of this permutation, and replace all $w_i$ for which $w_i \in C_1$ with $\star$'s. Call the resulting $\star$-permutation $s'$. Now, delete any $\star$'s in $s'$ which occur after the $(k-1)$-th $\star$-descent of $s'$. The result is a $\star$-permutation of $C_{k-1} \cup \{\lambda\} \cup C_k$ satisfying the desired conditions.
\end{proof}

\begin{cor}
We have
\[
A_{n-m, 0^{k-3}, r, m-r, 0^{n-k}} = \sum_{i=0}^{n-m} \binom{m+i}{m} A(m,k-i;r)
\]
where $A(n,k;r)$ equals the number of permutations $w \in S_{n+1}$ with $k-1$ descents and $w_1 = r+1$. In particular,
\[
A_{n-m,0^{k-2},m,0^{n-k}} = \sum_{i=0}^{n-m} \binom{m+i}{m} A(m,k-i)
\]
where $A(n,k)$ is defined to be 0 if $k \le 0$ or $k > n$.
\end{cor}

\section{Type $B$ mixed Eulerian numbers}

We now give an analogous combinatorial interpretation for the numbers $B_{c_1,\dotsc,c_n}$. Let $C = (C_1,\dotsc,C_n)$ be a division of a set $S$. We say that an element $s \in S$ is \emph{type $B$ admissible} with respect to $C$ if either $s$ is the smallest element of $C_1$ or $s \in C_i$ for $i \neq 1$. Given a type $B$ admissible element $s$, we now define the \emph{type $B$ deletion} of $s$ from $C$, which by abuse of notation we denote by $C^s$. Let $i$ be such that $s \in C_i$. If $i \neq n$, then we define $C^s$ to be the same as in the type $A$ case. If $i = n$, then we define
\[
C^s = (C_1,\dotsc,C_{n-2},C_{n-1} \cup (C_n \minus \{s\})).
\]
Given these definitions of admissibility and deletion, we define a \emph{type $B$ $C$-permutation} analogously as in the type $A$ case.

Recall that we defined
\[
B_{c_1,\dotsc,c_n} = n! \Vol(\Gamma_{1,n}^{c_1},\dotsc,\Gamma_{n,n}^{c_n}).
\]

\begin{thm} \label{mainB}
Let $C$ be a division. Then $B_{\abs{C}}$ equals $2^n$ times the number of type $B$ $C$-permutations.
\end{thm}

\begin{proof}
Since the proof is analogous to the type $A$ case, we will give an outline and leave details to the reader. Define
\begin{align*}
f_n(\lambda_1,\dotsc,\lambda_n) &= \Vol(\lambda_1\Gamma_{1,n} + \lambda_2\Gamma_{2,n} + \dotsb + \lambda_n\Gamma_{n,n}) \\
&= \sum_{c_1+\dotsb+c_n = n} \frac{1}{c_1!\dotsb c_n!} B_{c_1,\dotsc,c_n} \lambda_1^{c_1} \dotsb \lambda_n^{c_n}
\end{align*}
so that
\[
B_{c_1,\dotsc,c_n} = \partial_1^{c_1} \dotsb \partial_n^{c_n} f_n.
\]

We make the following observations, which are proven similarly to Proposition~\ref{propcross}, Corollary~\ref{corcross}, and Proposition~\ref{intrecursion}.

\begin{prop}
Let $y_1 \ge \dotsb \ge y_n \ge 0$ be real numbers, and let $SP = SP(y_1,\dotsc,y_n)$. Fix a real number $-y_1 \le x \le y_1$, and let $SP_x$ denote the cross section of $SP$ with first coordinate equal to $x$. Let $1 \le i \le n$ be such that $y_{i+1} \le \abs{x} \le y_i$, where we set $y_{n+1} = 0$. Then $SP_x$ is equal to
\[
\{x\} \times SP(y_1,\dotsc,y_{i-1},y_i+y_{i+1}-\abs{x},y_{i+1},\dotsc,y_n)
\]
if $i \le n-1$, and
\[
\{x\} \times SP(y_1,\dotsc,y_{n-1})
\]
if $i=n$.
\end{prop}

\begin{cor}
Let $\lambda_1$, \dots, $\lambda_n$ be nonnegative real numbers. Fix a real number $-(\lambda_1+\dotsb+\lambda_n) \le x \le \lambda_1+\dotsb+\lambda_n$, and let $1 \le i \le n$ be such that $\lambda_{i+1} + \dotsb + \lambda_n \le \abs{x} \le \lambda_i + \dotsb + \lambda_n$ (where $0 \le \abs{x} \le \lambda_n$ if $i=n$). Set $t = \lambda_i + \dotsb + \lambda_n - \abs{x}$. Then the cross section of
\[
\lambda_1 \Gamma_{1,n} + \lambda_1 \Gamma_{2,n} + \dotsb + \lambda_n \Gamma_{n,n}
\]
with first coordinate equal to $x$ is equal to $\{x\} \times Q$, where $Q$ is the following polytope in the following cases:
\begin{itemize}
\item If $i=1$,
\[
(t+\lambda_2)\Gamma_{1,n-1} + \lambda_3 \Gamma_{2,n-1} + \dotsb + \lambda_n \Gamma_{n-1,n-1}
\]
\item If $2 \le i \le n-1$,
\begin{multline*}
\lambda_1\Gamma_{1,n-1} + \dotsb + \lambda_{i-2}\Gamma_{i-2,n-1} + (\lambda_{i-1} + \lambda_i - t)\Gamma_{i-1,n-1} \\ + (t+\lambda_{i+1})\Gamma_{i,n-1} + \lambda_{i+2}\Gamma_{i+1,n-1} + \dotsb + \lambda_n\Gamma_{n-1},
\end{multline*}
\item If $i=n$,
\[
\lambda_1\Gamma_{i,n-1} + \dotsb + \lambda_{n-2}\Gamma_{n-2,n-1} + (\lambda_{n-1} + \lambda_n)\Gamma_{n-1,n-1}.
\]
\end{itemize}
\end{cor}

\begin{prop}
We have
\begin{align*}
f_n(\lambda_1,\dotsc,\lambda_n) ={ }& 2\int_0^{\lambda_1} f_{n-1}(t+\lambda_2,\lambda_3,\dotsc,\lambda_n)\,dt \\
& + 2\sum_{i = 2}^{n-1} \int_0^{\lambda_i} f_{n-1}(\lambda_1, \dotsc, \lambda_{i-2}, \lambda_{i-1}+\lambda_i-t, t+\lambda_{i+1}, \lambda_{i+2}, \dotsc, \lambda_n)\,dt \\
& + 2\int_0^{\lambda_n} f_{n-1}(\lambda_1,\dotsc,\lambda_{n-2},\lambda_{n-1}+\lambda_n)\,dt.
\end{align*}
\end{prop}

Differentiating this last equation, we obtain
\[
B_{c_1,\dotsc,c_n} = 2 \left( B_{\abs{C^1}} + \sum_{i=2}^{n-1}\sum_{s \in C_i} B_{\abs{C^s}} + \sum_{s \in C_n} B_{\abs{C^s}} \right)
\]
where $C$ is a division with $\abs{C} = (c_1,\dotsc,c_n)$ and all deletions are type $B$ deletions. The desired result follows by induction with the base case $B_1 = 2$.
\end{proof}

Using Theorem~\ref{mainB}, we obtain the following properties of type $B$ mixed Eulerian numbers. The proofs are similar to the type $A$ case and we omit them here.

\begin{thm} \label{propertiesB}
The type $B$ mixed Eulerian numbers have the following properties.
\renewcommand{\theenumi}{\alph{enumi}}
\renewcommand{\labelenumi}{(\theenumi)}
\begin{enumerate}
\item We have $2^n A_{c_1,\dotsc,c_n} \le B_{c_1,\dotsc,c_n} \le 2^n 1^{c_1} 2^{c_2} \dotsb n^{c_n}$. Each inequality is equality if and only if $c_1 + \dotsb + c_i \ge i$ for all $i$.
\item For $1 \le k \le n$, the number $B_{0^{k-1},n,0^{n-k}}$ is equal to $2^n$ times the number of permutations in $S_n$ with at most $k-1$ descents.
\item For $1 \le k \le n-1$ and $0 \le r \le n$, the number $B_{0^{k-1},r,n-r,0^{n-k-1}}$ is equal to $2^n$ times the number of permutations $w \in S_{n+1}$ with at most $k$ descents and $w_1 = r+1$.
\item We have $B_{1,\dotsc,1} = 2^n n!$.
\item We have $B_{k,0,\dotsc,0,n-k} = \binom{n}{k}(n-k)!$.
\item We have $B_{c_1,\dotsc,c_n} = 2^n 1^{c_1} 2^{c_2} \dotsb n^{c_n}$ if $c_1 + \dotsb + c_i \ge i$ for all $i$.
\item We have $B_{c_1,\dotsc,c_n} = 2^n n!$ if $c_n + c_{n-1} + \dotsb + c_{n-i+1} \ge i$ for all $i$.
\item We have
\[
B_{c_1,\dotsc,c_n} = 2^n \binom{n}{c_1+\dotsb+c_r} 1^{c_1} 2^{c_2} \dotsb r^{c_r} (c_{r+1} + \dotsb + c_n)!
\]
if there exists some $0 \le r \le n$ such that $c_1 + \dotsb + c_i \ge i$ for all $1 \le i \le r$ and $c_n + c_{n-1} + \dotsb + c_{n-i+1} \ge i$ for all $1 \le i \le n-r$.
\end{enumerate}
\end{thm}

\subsection*{Acknowledgments}
The author would like to thank Alexander Postnikov for introducing this problem to him and for many useful discussions. This material is based upon work supported by the National Science Foundation Graduate Research Fellowship under Grant No.\ 1122374.

\end{document}